\documentclass[reqno]{amsart}

\usepackage[utf8]{inputenc}
\usepackage{amssymb, amsmath, amsthm, color, enumerate, mathabx, mathrsfs}
\usepackage{bbm}
\usepackage{pdfpages}
\usepackage{tcolorbox}
\usepackage{esint} 
\usepackage{xifthen,ifthen}
\makeatletter
\newcounter{@ToDo}
\newcommand{\todo@helper}[1]{%
	({\color{blue}TODO~\arabic{@ToDo}: {#1\@addpunct{.}}})%
}
\newcommand{\todo}[1]{\stepcounter{@ToDo}%
	\relax\ifmmode\text{\todo@helper{#1}}%
	\else\todo@helper{#1}\fi%
}

\numberwithin{equation}{section}
\theoremstyle{plain}
\newtheorem{theorem}{Theorem}[section]
\newtheorem{lemma}[theorem]{Lemma}

\newtheorem{corollary}[theorem]{Corollary}

\theoremstyle{definition}

\newtheorem{remark}[theorem]{Remark}

\newcommand{\N}{\mathbb{N}}

\newcommand{\R}{\mathbb{R}}

\newcommand{\Z}{\mathbb{Z}}

\newcommand{\cC}{\mathcal{C}}

\newcommand{\cE}{\mathcal{E}}
\newcommand{\cF}{\mathcal{F}}

\newcommand{\cI}{\mathcal{I}}
\newcommand{\cJ}{\mathcal{J}}

\newcommand{\cL}{\mathcal{L}}
\newcommand{\cM}{\mathcal{M}}

\newcommand{\fkN}{\mathfrak{N}}

\newcommand{\bdJ}{\mathbf{J}}

\newcommand{\bde}{\mathbf{e}}

\newcommand{\bdj}{\mathbf{j}}

\newcommand{\ud}{\mathrm{d}}

\newcommand{\supp}{\mathrm{supp}\,}

\newcommand{\dist}{\mathrm{dist}}
\newcommand{\tvec}{T_0, \ldots, T_{n-2}}

\newcommand{\svec}{s_0, \ldots, s_{n-2}}

\begin{document}

\title[Integral points near space curves]{Counting integral points near space curves: a Fourier analytic approach}

\date{}

\begin{abstract} We establish upper and lower bounds for the number of integral points which lie within a neighbourhood of a smooth nondegenerate curve in $\mathbb{R}^n$ for $n\geq 3$. These estimates are new for $n\geq 4$, and we recover an earlier result of J. J. Huang for $n=3$. However, we do so by using Fourier analytic techniques which, in contrast with the method of Huang, do not require the sharp counting result for planar curves as an input. In particular, we rely on an Arkhipov--Chubarikov--Karatsuba-type oscillatory integral estimate.
\end{abstract}

\author[J. Hickman]{ Jonathan Hickman }
\address{School of Mathematics and Maxwell Institute for Mathematical Sciences, James Clerk Maxwell Building, The King's Buildings, Peter Guthrie Tait Road, Edinburgh, EH9 3FD, UK.}
\email{jonathan.hickman@ed.ac.uk}

\author[R. Srivastava]{ Rajula Srivastava }
\address{Mathematical Institute of the University of Bonn and Max Planck Institute for Mathematics, Bonn
Endenicher Allee 60, 53115, Bonn, Germany.}
\email{rajulas@math.uni-bonn.de}

\maketitle



\section{Introduction}



\subsection{Main result}

In this note, we establish upper and lower bounds for the number of integral points close to nondegenerate smooth curves in $\mathbb{R}^n$. To the best of our knowledge, our result is the first explicit non-trivial estimate for curves in dimensions four or greater. In the case $n=3$, we recover a result of J. J. Huang from \cite{huangintegral19} up to $\varepsilon$ losses, using Fourier analytic methods.\footnote{Huang remarks in \cite[p.1991]{huangintegral19} that, in principle, the argument used in \cite{huangintegral19} in $\R^3$ could be extended to obtain bounds for lattice points close to a curve in $\mathbb{R}^n$ with $n\ge4$. However, no explicit estimates are given and it is stated that the quality of the resulting estimates deteriorates rapidly as $n$ increases.}

Let $\cM$ be a compact $C^{\infty}$ curve in $\mathbb{R}^n$; that is, $\cM$ is the image of a regular, injective $C^{\infty}$ map $\gamma: I\to \mathbb{R}^n$, where $I$ is a compact interval. Our goal is to estimate the number of lattice points lying in a neighbourhood of a fixed dilate of $\cM$. More precisely, for $q\in \mathbb{N}$ and $0 < \delta < 1/2$, we let $N_{\cM}(q, \delta)$ denote the cardinality of the set
\begin{equation*}
\big\{ b \in \Z^n : \textrm{dist}\left(b/q, \cM\right)< \delta/q  \big\}.
\end{equation*}
Here $\textrm{dist}$ denotes the distance with respect to the Euclidean norm on $\mathbb{R}^n$.
Let $U\subset\mathbb{R}$ be an open set containing the interval concentric to $I$ but with twice the radius. We assume that $\gamma$ extends to a smooth map on $U$ which satisfies the \textit{non-degeneracy} condition
\begin{equation}
    \label{eq cond non-deg gam}
    \det \begin{bmatrix}
        \gamma^{(1)}(t)& \gamma^{(2)}(t) & \ldots & \gamma^{(n)}(t) 
    \end{bmatrix} \neq 0 \qquad \textrm{for all $t\in U$.}
\end{equation}
Finally, we introduce the exponent
\begin{equation}
    \label{def theta}
    \Theta(n):=\begin{cases}
       \displaystyle \frac{n^2+4}{n(n+4)} & \textrm{ for even } n,\\[8pt]
      \displaystyle  \frac{n^2+3}{n(n+4)-1} & \textrm{ for odd } n.
    \end{cases}
\end{equation}
With these definitions, our main result reads as follows. 
\begin{theorem}
\label{thm main}
For $n\geq 3$, let $\cM$ be as in \eqref{eq curve}. Let $\nu>0$. Suppose condition \eqref{eq cond non-deg gam} holds, and let $\Theta$ be as defined in \eqref{def theta}. Then for all $\delta\in (0, 1/4)$, we have
\begin{equation}
\label{eq main est}
{N}_{\mathcal{M}}(q, \delta)\ll_{\nu} \,\delta^{n-1} q+q^{\Theta(n)+\nu},
\end{equation}
where the implicit constant depends only on $\cM$ and $\nu$.
Further, for $q\gg_{\nu} 1$ and
\begin{equation}
    \label{eq delta range}
    \delta \geq \begin{cases}
    q^{-\frac{4}{n(n+4)}+\nu} & \textrm{ for even } n,\\
    q^{-\frac{4}{n(n+4)-1}+\nu} & \textrm{ for odd } n,    
    \end{cases}
\end{equation}
we also have the lower bound
\begin{equation}
    \label{eq main lb}
    {N}_{\mathcal{M}}(q, \delta)\gg_{\nu} \,\delta^{n-1} q.
\end{equation}
\end{theorem} 

Here, given a list of objects $L$ and non-negative real numbers $A$, $B$, we write $A \ll_L B$ or $B \gg_L A$ or $A = O_L(B)$ if $A \leq C_L B$ where $C_L > 0$ is a constant depending only on the objects listed in $L$, and possibly a choice of dimension $n$ and underlying curve $\cM$. We write $A \sim_L B$ when both $A \ll_L B$ and $B \gg_L A$.

\subsection{Background}

The problem of counting lattice/integral points near dilates of a fixed planar curve is well-studied, with sharp asymptotics known in great generality: see, for instance, \cite{huxley1996area, trifonov2002lattice, huang2020parab}. We remark in passing that this question is closely related to the classical Gauss circle problem. For the case where $\cM$ is a convex hypersurface, optimal bounds appear in the work of Lettington in three and higher dimensions \cite{lettington1, lettington2}.

Much less is known for space curves in $\R^n$ for $n \geq 3$. J. J. Huang \cite{huangintegral19} made the first non-trivial progress on the counting problem for integral points near nondegenerate curves in $\mathbb{R}^3$. In \cite{huangintegral19}, for $\cM$ a compact $C^{3}$ curve in $\mathbb{R}^3$ with non-vanishing curvature and torsion (in other words, satisfying condition \eqref{eq cond non-deg gam}), he established the upper bound
$$
N_{\cM}(q, \delta)\ll \delta^2q+q^{\frac35}(\log q)^{\frac45}.
$$
Further, he also proved a matching lower bound; that is, he showed that when $\delta \gg q^{-\frac15}(\log q)^{\frac25}$ and $q\gg 1$ we have
$$
N_{\cM}(q, \delta)\gg \delta^2q.
$$
The key innovation in his approach was an induction scheme that reduced the problem of counting integral points near space curves in three dimensions to the problem of counting \textit{rational points} near planar curves. Specifically, the argument exploits the optimal count for rational points near a planar curve $\cC$ with non-vanishing curvature,
\begin{equation}
    \label{eq rat pt planar}
    \sum_{q=1}^{Q}N_{\cC}(q, \delta)\ll_{\nu} \delta Q^{2}+ Q^{1+\nu}.
\end{equation}

The estimate \eqref{eq rat pt planar} for rational points near a sufficiently smooth planar curve $\cC$  with non-vanishing curvature was 
established by Vaughan--Velani in \cite{vaughan2006diophantine}. Their proof uses previous near-optimal estimates of Huxley \cite{huxley1994rational} for the same counting function as an input. In \cite{huang2018plancurves}, 
J. J. Huang built on these works to establish an asymptotic formula for this count. 

More generally, the problem of counting rational points near manifolds has seen rapid development in the recent years. Given a manifold $\cM$, one typically expects a significantly better bound for 
\begin{equation}
    \label{eq def rational pt}
    N_{\cM}^{\textrm{rat}}(Q, \delta):=\sum_{q=1}^Q N_{\cM}(q, \delta)
\end{equation}
than that obtained by trivially summing up the estimates for $N_{\cM}(q, \delta)$ for a fixed $q$. While interesting in its own right, this problem is also closely related to questions in Diophantine approximation and the dimension growth problem for submanifolds of $\mathbb{R}^n$ (see Section 1.2 and Conjecture 1.7 in \cite{rs24}). For recent results on the rational counting problem see \cite{Bers12, sst2023rational, huanghypers20, st2023, chen2025rational, by23, rs24}.

Finally, we remark in passing that the problem of counting integral points \textit{on} curves is classical and well studied. We refer the interested reader to \cite{jarnik1926gitterpunkte, swinnerton1974number, bombieripila, pila1991geometric, ellenberg2005uniform,heath2002density, walsh2015bounded} and the references therein.

\subsection{Conjectural bounds}
The moment curve, given by $\gamma_{\circ}(t):=(t,t^2, \ldots, t^n)$, is the prototypical example of a curve in $\mathbb{R}^n$ satisfying the non-degeneracy condition \eqref{eq cond non-deg gam}. For $\cM_{\circ} :=\{(t, t^2, \ldots, t^n): t\in [-1,1]\}$, there exist infinitely many $q \in \N$ such that
\begin{equation}\label{eq: easy count}
   N_{\cM_{\circ}}(q, 0) := q\cM_{\circ}\cap\mathbb{Z}^n\gg q^{\frac{1}{n}}. 
\end{equation}
Indeed, suppose $q$ of the form $q=z^n$ where $z\in \mathbb{N}$. Then for each $a\in \{0, 1, \ldots, z-1\}$, we have
$$q\left(\frac{a}{z}, \frac{a^2}{z^2}, \ldots, \frac{a^n}{z^n}\right)\in q\cM_{\circ}\cap\mathbb{Z}^n.$$

In view of \eqref{eq: easy count}, for any $\delta\in (0, 1/4)$, we automatically have $N_{\cM_{\circ}}(q, \delta)\gg q^{\frac{1}{n}}$ for infinitely many $q \in \N$. It is therefore natural to conjecture that
\begin{equation*}
    N_{\cM_{\circ}}(q, \delta)\ll \delta^{n-1}q +  q^{\frac{1}{n}} \qquad \textrm{for $q\geq 1$ and $\delta\in (0, 1/4)$.}
\end{equation*}
Furthermore, applying the same reasoning as above, a lower bound of the form
\begin{equation*}
    N_{\cM_{\circ}}(q, \delta)\gg \delta^{n-1} q
\end{equation*}
can only be true in the range $\delta> q^{-\frac{1}{n}}$, at least when $q$ is an $n$th power. Thus, we expect both the upper bound \eqref{eq main est} from Theorem~\ref{thm main}, as well as the range of $\delta$ in \eqref{eq delta range} for the lower bound, are far from sharp. 

As mentioned before, one expects the estimates for the number of rational points near $\cM_{\circ}$, as defined in \eqref{eq def rational pt}, to be even better due to the additional cancellation in the $q$ variable. Moreover, for a curve $\cM$ in $\mathbb{R}^n$ satisfying \eqref{eq cond non-deg gam}, it is widely conjectured that, for all $\varepsilon>0$, the bound
\begin{equation*}
    N_{\cM}^{\textrm{rat}}(Q, \delta)\ll_{\varepsilon} \delta^{n-1}Q^2 +  Q^{1+\varepsilon},
\end{equation*}
holds for any $Q\geq 1$, $\delta\in (0, 1/4)$. Indeed, the aforementioned upper bound would be a corollary of \cite[Conjecture 3.1]{huang2024extremal}, formulated for manifolds of arbitrary codimension under a more general non-degeneracy condition. When $\cM$ is a curve in $\mathbb{R}^n$ satisfying \eqref{eq cond non-deg gam}, this conjecture states that, for all $\nu>0$, the bound
$$N_{\cM}^{\textrm{rat}}(Q, \delta)\ll \delta^{n-1}Q^2 \qquad \textrm{ holds for } \delta\geq Q^{-\frac{1}{n-1}+\nu} \textrm{ and } Q\to \infty.
$$
Observe that for $\delta\leq Q^{-\frac{1}{n-1}+\nu}$, the above combined with the monotonicity of the counting function in $\delta$ would imply that
$$N_{\cM}^{\textrm{rat}}(Q, \delta)\leq N_{\cM}^{\textrm{rat}}(Q, Q^{-\frac{1}{n-1}+\nu})\ll Q^{1+\varepsilon}, \qquad \textrm{ with } \varepsilon=(n-1)\nu \textrm{ and as } Q\to \infty.
$$


In the other direction, in a breakthrough work \cite{Bers12}, Beresnevich established the expected lower bound  
\begin{equation*}
    N_{\cM}^{\textrm{rat}}(Q, \delta)\gg \delta^{n-1}Q^2
\end{equation*}
for analytic, nondegenerate curves in the (wider) range $\delta>Q^{-\frac{3}{2n-1}}$. This result was recently extended to smooth, nondegenerate curves in \cite{sst2023rational}, up to $Q^{\varepsilon}$ losses.

\subsection{Proof outline and structure of the paper}
Our main objective is to provide a streamlined Fourier analytic approach to estimating integral points near nondegenerate space curves. The first few steps are standard. Using Poisson summation, we express our lattice point count first as an exponential sum and then as a sum of oscillatory integrals sampled over frequencies arising from an $n$-dimensional lattice. More precisely, our oscillatory integrals have the form
\begin{equation}
    \label{eq osc int form}
    \int_{-1}^{1}  e^{2\pi i \gamma(t)\cdot\bdj} \,\ud t,
\end{equation}
where $\bdj\in \mathbb{Z}^n$ is a frequency vector from an $n$-dimensional lattice, satisfying $|\bdj|\ll \delta^{-1}$. We isolate the zero frequency term, corresponding to $\bdj = \mathbf{0}$, which yields the expected main contribution of order $\delta^{n-1}q$. This is carried out in Section~\ref{sec ini reduct}.

We refer to the total contributions from the nonzero frequencies $\bdj\in \mathbb{Z}^n\setminus\{0\}$ as the `error'. Matters are reduced to showing that the error is $o(\delta^{n-1}q)$ for a non-trivial range of $\delta$. The strategy is to use the non-degeneracy condition \eqref{eq cond non-deg gam} to extract decay from the oscillatory integrals \eqref{eq osc int form}. The classical van der Corput estimate for oscillatory integrals (see, for example, \cite[Chapter VIII]{Stein_book}) guarantees 
\begin{equation}
    \label{eq worst decay}
    \left|\int_{-1}^{1}  e^{2\pi i \gamma(t)\cdot\bdj} \,\ud t\right|\ll (1 + |\bdj\,|)^{-1/n}.
\end{equation}
Summing this over the set of frequencies satisfying $0<|\bdj|\ll \delta^{-1}$, it is not hard to see that the error is admissible in the restricted range
$\delta>q^{-\frac{1}{n^2-1}+\varepsilon}$. 

We can improve over the basic van der Corput bound described above by exploiting the fact that the decay rate in \eqref{eq worst decay} is the worst possible. The estimate \eqref{eq worst decay} is optimal only in the case when the first $n-1$ derivatives of the phase function $t \mapsto \gamma(t)\cdot\bdj$ vanish (or are close to $0$). This is only true for a very small number of frequencies $\bdj$, concentrated near a two-dimensional cone in the frequency space $\widehat{\R}^n$. By keeping track of the exact order of vanishing of the derivatives of the phase function, we can typically extract more decay from the associated oscillatory integral. The \textit{H-functional}, as defined in \eqref{eq def h func}, is a device that allows us to do this. In Theorem~\ref{thm: ACK}, under hypothesis \eqref{eq cond non-deg gam}, we present decay estimates for oscillatory integrals of the form \eqref{eq osc int form} in terms of the $H$-functional. These refined oscillatory integral estimates essentially appear in \cite{BGGIST2007} and are modelled after the classical work of Arkhipov--Chubarikov--Karatsuba \cite[p.13, Theorem 1.1]{ACK_book}. They allow us to estimate the error in terms of the size of the $H$-functional, ranging from the worst possible decay of the order of $(1 + |\bdj\,|)^{-1/n}$ to the strong decay rate of $(1 + |\bdj\,|)^{-1/2}$, with the set of contributing frequencies decomposed as level sets of the functional. This is the content of Section~\ref{sec osc int est}.

The next crucial ingredient is to reinterpret the sublevel sets of the $H$-functional as neighbourhoods of a two-dimensional cone $\Gamma^2$ in $\widehat{\R}^n$, which corresponds to the region of worst possible decay for our oscillatory integrals. This is carried out in Section~\ref{sec sublevel set H func}. This interpretation enables us to estimate the cardinality of the $\bdj \in \Z^n$ lying in the sublevel sets of the $H$-functional by counting lattice points lying in a neighbourhood of $\Gamma^2$. For the purposes of our argument, here we use trivial counting results based on volume estimates. We see that the oscillatory integral decays slowly for a relatively sparse set of $\bdj$. Putting these estimates together, we can bound the error in a more efficient way and therefore establish Theorem~\ref{thm main} in a better range of $\delta$ than that implied by just the van der Corput estimate. This is carried out in Section~\ref{sec proof main}.\medskip  

This paper is intended as a proof of concept. In particular, we relate the problem of counting lattice points near the curve $\gamma$ in $\mathbb{R}^n$ to a system of dual counting problems associated to neighbourhoods of nested surfaces in the frequency space $\widehat{\R}^n$. In principle, non-trivial estimates for the counting function associated to these dual surfaces should lead to better estimates for the original counting function associated to $\gamma$. Furthermore, one could attempt to exploit cancellation between the oscillatory integrals \eqref{eq osc int form}, in the spirit of the van der Corput A/B processes. However, there are substantial complications when attempting to implement these ideas. Nevertheless, we believe this is a promising direction for future research.

\subsection{Acknowledgement} The first author is supported by New Investigator Award UKRI097. The second author is supported by an Argelander grant from the University of Bonn. Both authors thank Jim Wright and the anonymous referee for helpful comments which improved the exposition of the article.

\section{Initial Reductions}
\label{sec ini reduct}






\subsection{A smooth variant of the counting function}
\label{ss smooth}
Our curve $\cM$ can be locally parametrised as the graph of a $C^{\infty}$ function of a single variable. Using compactness to cover $\cM$ by a finite number of local charts arising from such a parametrisation, and then by scaling, translation and a relabelling of the coordinates, we may assume without loss of generality that $\cM$ has the representation
\begin{equation}
    \label{eq curve}
    \cM := \big\{(t, f(t)) : t \in [-1,1] \big\}
\end{equation}
where $f:=(f_2, \ldots, f_n) \colon [-1,1] \to \R^{n-1}$ and $f_2, \ldots, f_n$ are restrictions of smooth functions defined on an open neighbourhood $U$ of $[-2,2]$. 

Having made the above reduction, it is not difficult to see that $N_{\cM}(q, \delta)$ is bounded above and below by a constant (depending on $\cM$) times the cardinality of the set
\begin{equation*}
\#\big\{ a \in \Z  : |a|\leq q,  \, \|qf(a/q)\| < \delta  \big\},
\end{equation*}
where $\|x\| := \dist(x, \Z^{n-1})$ for $x \in \R^{n-1}$.

Further, the non-degeneracy condition \eqref{eq cond non-deg gam} now reads
\begin{equation}
    \label{eq cond non-deg}
    \det \begin{bmatrix}
        1& f_2^{(1)}(t) & \ldots & f_n^{(1)}(t) \\
        0 & f_2^{(2)}(t) & \ldots & f_n^{(2)}(t) \\
        \vdots \\
        0 & f_2^{(n)}(t) & \ldots & f_n^{(n)}(t)
    \end{bmatrix} \neq 0 \qquad \textrm{for all $t \in U$.}
\end{equation}

Following \cite{huanghypers20, sst2023rational, st2023, rs24}, we shall work with a smooth version of the counting function $N_{\cM}(q, \delta)$.\medskip

Let $w^{\pm}\in C^{\infty}_c(\R)$ take values in $[0,1]$ and satisfy
\begin{itemize}
    \item $\supp w^- \subseteq [-1,1]$ and $w^-(t) = 1$ for all $t \in [-1/2,1/2]$;
    \item $\supp w^+ \subseteq [-2, 2]$ and $w^+(t) = 1$ for all $t \in [-1,1]$.
\end{itemize}
Similarly, let $\eta^{\pm}\in C^{\infty}_c(\R^{n-1})$ take values in $[0,1]$ and satisfy
\begin{itemize}
    \item $\supp \eta^- \subseteq [-1,1]^{n-1}$ and $\eta^-(y) = 1$ for all $y \in [-1/2,1/2]^{n-1}$;
    \item $\supp \eta^+ \subseteq [-2, 2]^{n-1}$ and $\eta^+(y) = 1$ for all $y \in [-1,1]^{n-1}$. 
\end{itemize}

For $0 < \delta < 1/4$, let $\Lambda(\delta)$ denote the set of all $x \in \R^{n-1}$ such that $\|x\| \leq \delta$; in other words, $\Lambda(\delta)$ is the $\delta$ neighbourhood of the integer lattice $\mathbb{Z}^{n-1}$. The characteristic function $\chi_{\Lambda(\delta)}$ of $\Lambda(\delta)$ then satisfies
$$\sum_{b\in \mathbb{Z}^{n-1}}\eta^{-}(\delta^{-1}(x-b))\leq \chi_{\Lambda(\delta)}(x)\leq \sum_{b\in \mathbb{Z}^{n-1}}\eta^+(\delta^{-1}(x-b))\qquad \textrm{for all $x \in {\R}^{n-1}$}.$$

For $w \in \{w^-, w^+\}$, $\eta\in \{\eta^-, \eta^+\}$ and $\delta\in \left(0, \frac{1}{4}\right)$, we define 
\begin{equation*}
    \fkN_{\cM}(q, \delta) = \fkN_{w, \eta, \cM}(q, \delta):=\sum_{{a\in \mathbb{Z}}}\sum_{b\in \mathbb{Z}^{n-1}}w(a/q)\, \eta\big(\delta^{-1}(qf(a/q)-b)\big).
\end{equation*}
We shall prove \eqref{eq main est} and \eqref{eq main lb} for $\fkN_{\cM}(q, \delta)$ instead of ${N}_{\mathcal{M}}(q, \delta)$. The required estimates for  ${N}_{\mathcal{M}}(q, \delta)$ then follow by approximating the characteristic functions of the interval $[-1,1]$ and the set $\Lambda(\delta)$ by the smooth weight functions $w^{\pm}$ and $\sum_{b\in \mathbb{Z}^{n-1}}\eta^{\pm}(\delta^{-1}(\cdot-b))$ respectively. 

\subsection{Exponential sum reformulation}  
Using the Poisson summation formula for $\eta\in\{\eta^+, \eta^-\}$, we get
\begin{equation}\label{eq: Poisson 2}
    \sum_{b \in \Z^{n-1}} \eta(\delta^{-1}(x-b)) = \delta^{n-1}\sum_{j \in \Z^{n-1}} \widehat{\eta}(\delta j) e^{2 \pi i x \cdot j}.
\end{equation}
Here $\widehat{\eta}$ denotes the Fourier transform of $\eta$ appropriately normalised. Using \eqref{eq: Poisson 2}, we can represent $\fkN_{\cM}(q, \delta)$ in terms of an exponential sum
\begin{equation}
\label{eq Nsm exp}
\fkN_{\cM}(q, \delta) = \delta^{n-1} \sum_{j \in \Z^{n-1}} \widehat{\eta}(\delta j)  \sum_{a \in \Z}  e^{2 \pi i qf(a/q) \cdot j} w(a/q).
\end{equation}
The above is an infinite sum in the variable $j$, however we can use the Schwartz decay of $\widehat{\eta}$ to conclude that, for any $\varepsilon > 0$, we have
\begin{equation*}
    \sum_{j \in \Z^{n-1}} \widehat{\eta}(\delta j)  e^{2 \pi i qf(a/q) \cdot j}= \sum_{\| j \|_{\infty} 
\leq q^{\varepsilon}/\delta} \widehat{\eta}(\delta j)  e^{2 \pi i qf(a/q) \cdot j}+ O_{\varepsilon}\left(\delta^{-(n-1)}q^{-1}\right).
\end{equation*}
Plugging the above into \eqref{eq Nsm exp} and isolating the frequency $j=0$, we get
\begin{equation}
\label{eq red 1}
\fkN_{\cM}(q, \delta)= c_0\delta^{n-1}q + \delta^{n-1}\sum_{1 \leq\| j \|_{\infty} 
\leq q^{\varepsilon}/\delta} \widehat{\eta}(\delta j)  \sum_{a \in \Z}  e^{2 \pi i qf(a/q) \cdot j} w(a/q) + O_{\varepsilon}(1).
\end{equation}
Here $c_0 := \hat{\eta}(0)\cdot q^{-1}\sum_{a\in\mathbb{Z}}w(a/q)$ satisfies $c_0 \sim 1$.




\subsection{From exponential sums to oscillatory integrals} For a fixed $j\in \mathbb{Z}^{n-1}$ with $1 \leq\| j \|_{\infty} \leq q^{\varepsilon}/\delta$, we temporarily focus on  the exponential sum
\begin{equation*}
     \sum_{a \in \Z}  e^{2 \pi i qf(a/q) \cdot j} w(a/q).
\end{equation*}
Our goal is to pass from this exponential sum to a corresponding oscillatory integral, which will be achieved by another application of the Poisson summation formula and integration by parts.

First, Poisson summation and a change of variables yield 
\begin{equation}
\label{eq pois 2}
  \sum_{a \in \Z} e^{2\pi i qf(a/q) \cdot j} w(a/q) = q \sum_{k \in \Z} \int_{\R} e^{2\pi i q(f(t) \cdot j - tk)} w(t)\,\ud t .
\end{equation}
For $\|j\|_{\infty} \leq q^{\varepsilon}\delta^{-1}$, we have $|\partial_t f(t)\cdot j| \leq M_\gamma q^{\varepsilon}\delta^{-1}/2$ for all $t \in \supp w$, where 
\begin{equation*}
    M_\gamma := 2\sup\{\|\partial_tf(t)\|_1 : t \in [-2,2]\}.
\end{equation*}
Non-stationary phase shows that most terms in the sum in $k$ only contribute a negligible error.
\begin{lemma}
For $\|j\|_{\infty} \leq q^{\varepsilon}\delta^{-1}$, we have
\begin{equation}\label{eq: Poisson 5}
     \sum_{|k|\geq  M_\gamma q^{\varepsilon}\delta^{-1}} \int_{\R} e^{2\pi i q(f(t) \cdot j - tk)} w(t)\,\ud t  =  O(q^{-n}).
\end{equation}    
\end{lemma}
\begin{proof}
For $k \in \Z$ satisfying $\pm k\geq  M_\gamma q^{\varepsilon}\delta^{-1}$, we let
\begin{equation*}
 \psi(t):=\frac{f(t)\cdot j-tk}{|k \mp M_\gamma q^{\varepsilon}\delta^{-1}/2|} \qquad \textrm{so that} \qquad
|\psi'(t)|=\frac{|k-f'(t)\cdot j|}{|k \mp M_\gamma q^{\varepsilon}\delta^{-1}/2|}\geq 1.   
\end{equation*}
Furthermore, 
$$|k\mp M_\gamma q^{\varepsilon}\delta^{-1}/2|\geq \max\left\{|f'(t)\cdot j|, \frac{|k|}{2}\right\}.$$
Therefore, we also have the upper bound $|\psi'(t)|\ll 1$. It is straightforward to check that the higher order derivatives of $\psi$ are bounded as well, by constants depending only on $f$. Thus, by repeated integration-by-parts we may estimate
\begin{equation}
    \label{eq k est nonst}
    \left|\int_{\R} e^{2\pi i q(f(t) \cdot j - tk)} w(t)\,\ud t\right|\ll \left(q|k \mp M_\gamma q^{\varepsilon}\delta^{-1}/2|\right)^{-n}.
\end{equation}
Summing up \eqref{eq k est nonst} in the range $|k|\geq M_\gamma q^{\varepsilon}\delta^{-1}$, we conclude the proof.
\end{proof}

Introducing the notation
\begin{equation*}
    \cI_{\cM}(q; k, j) := \int_{\R} e^{2\pi i q(f(t) \cdot j - t \cdot k)} w(t)\,\ud t,
\end{equation*}
we may combine \eqref{eq red 1}, \eqref{eq pois 2} and \eqref{eq: Poisson 5} to conclude that
\begin{equation}
    \label{eq red 2}
   \fkN_{\cM}(q, \delta) = c_0\,\delta^{n-1}q + \delta^{n-1} q \sum_{\substack{j \in \Z^{n-1}\setminus\{0\}:\\ \|j\|_{\infty} \leq q^{\varepsilon}\delta^{-1}} } \widehat{\eta}(\delta j)  \sum_{\substack{k \in \Z \\ |k| \leq M_\gamma q^{\varepsilon} \delta^{-1}}} \cI_{\cM}(q; k, j)  +  O_{\varepsilon}(1).
\end{equation}
We group the indices $j$ and $k$ together to write $\bdj = (k, j)\in \mathbb{Z}\times\mathbb{Z}^{n-1}$, and let
\begin{equation*}
    \bdJ=\bdJ(q, \delta):=\{(k, j)\in \mathbb{Z}\times \mathbb{Z}^{n-1}: 0<\|j\|_{\infty}\leq q^{\varepsilon}\delta^{-1}, |k|\leq M_\gamma q^{\varepsilon}\delta^{-1}\}.
\end{equation*}
We now focus on estimating
\begin{equation}
\label{eq def sigma}
  \Sigma_{\cM}(q, \delta) := \delta^{n-1} q  \,\sum_{\bdj\in \bdJ} \widehat{\eta}(\delta j)\, \cI_{\cM}(q; \bdj).
\end{equation}
We shall deduce a crude bound for the above sum by applying the triangle inequality and estimating the $|\cI_{\cM}(q; \bdj)|$ individually.




\section{Oscillatory integral estimates}
\label{sec osc int est}







\subsection{An Arkhipov--Chubarikov--Karatsuba-type estimate} For a compact interval $I$ and $\gamma \colon I \to \R^n$ a parametrisation of a smooth curve, we define the \textit{$H$-functional} by
\begin{equation}
    \label{eq def h func}
    H_{\gamma}(\xi) := \inf_{t \in I} \max_{1 \leq r \leq n} \big|\gamma^{(r)}(t) \cdot \xi \big|^{1/r} \qquad \textrm{for all $\xi \in \widehat{\R}^n$.}
\end{equation}

\begin{theorem}[Arkhipov--Chubarikov--Karatsuba-type estimate] 
\label{thm: ACK}
Let $I\subseteq [-2, 2]$ be a compact interval, $\gamma \colon I \to \R^n$ be a smooth curve and  let $\xi \in \widehat{\R}^n$ satisfy 
\begin{equation}
    \label{eq der lb}
    \inf_{t\in I}|\gamma^{(\ell)}(t)\cdot \xi| > c|\xi|
\end{equation}
for some $c > 0$ and $1\leq \ell\leq n$. Then we have 
\begin{equation}
    \label{eq ACK}
    \Big|\int_I e^{2\pi i \gamma(t) \cdot \xi} \,\ud t\Big| \ll_{c, \gamma} (1 + H_{\gamma}(\xi))^{-1}. 
\end{equation}
\end{theorem}

An exposition of the above theorem in the special case where $\gamma = \gamma_{\circ}$ is the moment curve be found in the monograph of Arkhipov--Chubarikov--Karatsuba \cite[p.13, Theorem 1.1]{ACK_book}. Taking $\gamma = \gamma_{\circ}$ corresponds to studying oscillatory integrals with phase functions $\gamma_{\circ}(t) \cdot \xi = \xi_1 t + \cdots + \xi_n t^n$ given by general univariate polynomials. Recently, dall'Ara--Wright \cite{dAW} gave an interesting new proof of this classical case, which also extends to polynomials over general local fields. Furthermore, Arkhipov--Chubarikov--Karatsuba \cite[Theorem 1.1$^\prime$]{ACK_book} also consider oscillatory integrals with non-polynomial phases, under an additional monotonicity hypothesis on the derivatives of the phase; this result is closely related to Theorem~\ref{thm: ACK} above. Theorem~\ref{thm: ACK} is a very precise and arguably definitive estimate, at least for the class of univariate polynomial phases. Its sharpness is discussed in detail in \cite{ACK_book, dAW} and in \cite{dAW} it is shown to imply many other results in the oscillatory integral literature. 

Estimates of the form \eqref{eq ACK} for general smooth functions satisfying \eqref{eq cond non-deg gam} were considered in \cite[Section 4]{BGGIST2007}, where the authors remark that they are immediate from van der Corput's lemma. To ensure that our exposition is self-contained, however, we include a detailed proof.

\begin{proof}[Proof (of Theorem~\ref{thm: ACK})] Fix $\xi \in \widehat{\R}^n$ satisfying $|\gamma^{(\ell)}(t) \cdot \xi| > c|\xi|$ for all $t \in I$, for some $c > 0$ and $1\leq \ell\leq n$. Define $\beta_m: I\to \mathbb{R}$ by $\beta_m(t) := \gamma^{(m)}(t) \cdot \xi$ for $t \in I$ and all $1 \leq m \leq n$. From our hypotheses, the function $\beta_\ell$ is nonvanishing on $I$. 
\medskip

We first consider the regime when $\ell \geq 2$. We claim that for each $1 \leq r \leq \ell-1$, there exists a collection $\cI_r$ of at most $2^{\ell-r-1}$ closed intervals which form a partition of $I$ and are such that the map $\beta_m$ is strictly monotone on $I_r$ for all $I_r \in \cI_r$ and all $r \leq m \leq \ell-1$. We prove this by induction.\medskip

Our earlier observations imply that $\beta_{\ell-1}$ has no critical points in $I$, and so this function is strictly monotone on the whole interval. Thus, we can take $\cI_{\ell-1} := \{I\}$, which establishes the base case.\medskip

Let $2 \leq r \leq \ell-1$ and assume, as an induction hypothesis, that there exists a partition $\cI_r$ satisfying the desired properties. Thus, given any $I_r \in \cI_r$, we know that $\beta_r$ is strictly monotone on $I_r$ and so can have at most one zero on this interval. Thus, $\beta_{r-1}$ has at most one critical point on $I_r$ and so we may partition this interval into two subintervals such that $\beta_{r-1}$ is strictly monotone on each. Let $\cI_{r-1}$ denote the collection of all intervals formed in this way. Thus, $\cI_{r-1}$ forms a partition of $I$ into at most $2 \cdot \#\cI_r \leq 2^{\ell-r}$ closed intervals, and each of the functions $\beta_m$ is strictly monotone on $I_{r-1}$ for all $I_{r-1} \in \cI_{r-1}$ and all $r-1 \leq m \leq \ell-1$. This closes the induction.\medskip

We further refine the partition $\cI_1$ as follows. Given any $I_1 \in \cI_1$, since the functions $\beta_r$ are strictly monotone on $I_1$ for $1 \leq r \leq \ell-1$, each set
\begin{equation}\label{eq: ACK 1}
\big\{t \in I_1 : |\beta_r(t)| < H_{\gamma}(\xi)^r \big\}
\end{equation}
is an open subinterval of $I_1$ and each set
\begin{equation}\label{eq: ACK 2}
\big\{t \in I_1 : |\beta_r(t)| \geq H_{\gamma}(\xi)^r \big\}
\end{equation}
is a union of at most two disjoint closed subintervals of $I_1$. Let $\cJ_r(I_1)$ denote the collection of closed intervals formed by taking the closure of \eqref{eq: ACK 1} and the constituent intervals of \eqref{eq: ACK 2}. Thus, $\cJ_r(I_1)$ is a partition of $I_1$ into at most three closed subintervals. We then let $\cJ(I_1)$ denote a common refinement of the partitions $\cJ_r(I_1)$ over all $1 \leq r \leq \ell-1$. Finally, we let $\cJ$ denote the union of the $\cJ(I_1)$ over all $I_1 \in \cI_1$.\medskip

The collection $\cJ$ has the following properties:
\begin{enumerate}[i)]
    \item $\cJ$ is a partition of $I$ into (essentially disjoint) closed subintervals;
    \item $\#\cJ = O_n(1)$;
    \item For each $J \in \cJ$ and each $1 \leq r \leq \ell-1$, either $|\beta_r(t)| \geq H_{\gamma}(\xi)^r$ for all $t \in J$ or $|\beta_r(t)| < H_{\gamma}(\xi)^r$ for all $t \in \mathrm{int}\,J$;
    \item $\beta_1$ is monotone on $J$ for each $J \in \cJ$.
\end{enumerate}

Fixing $J \in \cJ$, by the triangle inequality it suffices to show 
\begin{equation}\label{eq: ACK 3}
    \Big|\int_J e^{2\pi i \gamma(t) \cdot \xi} \,\ud t\Big| \ll (1 + H_{\gamma}(\xi))^{-1}.
\end{equation}
Suppose there exists some $1 \leq r \leq \ell-1$ such that $|\beta_r(t)| \geq H_{\gamma}(\xi)^r$ for all $t \in J$. The desired inequality \eqref{eq: ACK 3} then follows by van der Corput's lemma with $r$ derivatives (see, for example, \cite[Chapter VIII, Proposition 2]{Stein_book}). Moreover, note that property iv) of the decomposition ensures that we may also apply van der Corput in the $r = 1$ case.\medskip

In light of the above observation and the properties of the decomposition, we may assume $|\beta_r(t)| < H_{\gamma}(\xi)^r$ for all $t \in \mathrm{int}\,J$ and all $1 \leq r \leq \ell-1$. Fixing $t_0 \in \mathrm{int}\,J$, from the definition of the $H$-functional, there exists some $1 \leq r_0 \leq n$ such that 
\begin{equation*}
    H_{\gamma}(\xi) \leq \big| \gamma^{(r_0)}(t_0) \cdot \xi \big|^{1/r_0} = |\beta_{r_0}(t_0)|^{1/r_0}.
\end{equation*}
However, from our earlier assumption this forces $r_0 \geq \ell$. We have the chain of inequalities 
$$1+H_{\gamma}(\xi) \leq 1+\big| \gamma^{(r_0)}(t_0) \cdot \xi \big|^{1/r_0}\ll_{\gamma} \left(1+|\xi|\right)^{1/r_0}\leq  (1+|\xi|)^{1/\ell}.$$
On the other hand, by our hypothesis \eqref{eq der lb} and van der Corput's inequality with $\ell$ derivatives,
\begin{equation*}
  \Big|\int_J e^{2\pi i \gamma(t)\cdot\xi} \,\ud t\Big| \ll_c (1 + |\xi\,|)^{-1/\ell}.
\end{equation*}
Combining these observations, we conclude that 
$$\Big|\int_{J} e^{2\pi i \gamma(t)\cdot\xi} \,\ud t\Big| \ll_c (1 + |\xi\,|)^{-1/\ell}\ll_{\gamma} (1+H_{\gamma}(\xi))^{-1}, $$
as required. \medskip

The case when $\ell=1$ is much simpler. In particular, integration-by-parts gives
\begin{equation*}
  \Big|\int_I e^{2\pi i \gamma(t)\cdot\xi} \,\ud t\Big| \ll_{c, \gamma} (1 + |\xi|)^{-1}.
\end{equation*}
Since $$1+H_{\gamma}(\xi)\ll_{\gamma} 1+\max_{1 \leq r \leq n} \big|\xi \big|^{1/r}\ll 1+|\xi|,$$
the desired inequality follows.
\end{proof}

\begin{corollary}
    \label{cor sub dec}
    Let $\gamma \colon [-2,2] \to \R^n$ be a smooth curve satisfying \eqref{eq cond non-deg gam} and $w \in \{w^-, w^+\}$. For $\xi \in \widehat{\R}^n\setminus\{0\}$, we have 
\begin{equation*}
    \Big|\int_{\R} e^{2\pi i \gamma(t) \cdot \xi} w(t)\,\ud t\Big| \ll_{\gamma} (1 + H_{\gamma}(\xi))^{-1}. 
\end{equation*}
\end{corollary}

Under the hypothesis of Corollary~\ref{cor sub dec}, an elementary argument shows that
\begin{equation}\label{eq: H functional coefficients}
    (1 + |\xi|)^{1/n} \ll 1 + H_{\gamma}(\xi).
\end{equation}
In this way, Corollary~\ref{cor sub dec} is (modulo uniformity in $\gamma$) a strengthening of the basic estimate following from van der Corput's lemma (see \eqref{eq worst decay}). 

\begin{proof}[Proof (of Corollary~\ref{cor sub dec})]
The condition \eqref{eq cond non-deg gam} ensures that there exists some $c > 0$, depending only on $\gamma$, such that  
\begin{equation*}
\sum_{r=1}^{n}|\gamma^{(r)}(t)\cdot\xi| > c|\xi| \qquad \textrm{for all $\xi \in \widehat{\R}^n \setminus\{0\}$ and all $t\in [-2,2]$.}
\end{equation*}
Using the continuity of the derivatives and compactness of both the set of unit vectors $S^{n-1}$ in $\mathbb{R}^{n}$ and the interval $[-2, 2]$, we can find a number $M>0$, depending only on $\gamma$ and $n$, such that the following holds. Given any $\xi\in\mathbb{R}^n \setminus \{0\}$, the interval $[-2,2]$ can be decomposed into at most $M$ subintervals such that \eqref{eq der lb} is true for some $1\leq \ell\leq n$ on each subinterval (with a smaller value of $c$ than that above). The desired inequality now follows by applying Theorem~\ref{thm: ACK} separately on each subinterval, together with a standard integration by parts argument to deal with the smooth cutoff $w$ (see, for example, \cite[p. 334]{Stein_book}), and the triangle inequality.
\end{proof}




\subsection{Level set decomposition} 
We now set $\gamma(t):=(t, f(t))$, where $f$ is as in the local parametrisation \eqref{eq curve} of $\cM$ and satisfies \eqref{eq cond non-deg}. We estimate the right-hand side of \eqref{eq def sigma} using Corollary~\ref{cor sub dec} to obtain
\begin{equation*}
    \Sigma_{\cM}(q, \delta) \ll \delta^{n-1} q  \,\sum_{\bdj\in \bdJ} \widehat{\eta}(\delta j)\, H(q; \bdj)^{-1}
\end{equation*}
where
\begin{equation*}
  H(q;\bdj):=H_{\gamma}(q\bdj).  
\end{equation*}
We decompose the above sum in $\bdj$ dyadically, giving
\begin{equation*}
  \Sigma_{\cM}(q, \delta) \ll \delta^{n-1}  q \sum_{\substack{1 \leq R \leq 2M_\gamma q^{\varepsilon}\delta^{-1} \\ R\,\mathrm{dyadic}}}\sum_{\substack{\bdj \in \Z^n \\ R/2 \leq |\,\bdj\,| < R}}    H(q; \bdj)^{-1};
\end{equation*}
here, as the notation suggests, the sum in $R$ is over all integers of the form $R = 2^k$ for $k \in \N_0$ with $0 \leq k \leq \log_2 (M_\gamma q^{\varepsilon}\delta^{-1})+1$.

Given $R, \lambda > 0$ and $q\in \mathbb{Z}$ with $q\geq 1$, define the level sets
\begin{equation}\label{eq: sublevel}
    S^{q,R}( \lambda) := \big\{\bdj \in \bdJ  : R/2 \leq |\,\bdj\,| < R \textrm{ and } \lambda/2 \leq H(q; \bdj) < \lambda\big\}.
\end{equation}
It follows from \eqref{eq: H functional coefficients} that there exists some $c > 0$, depending only $n$ and $\gamma$, such that
\begin{equation}
\label{eq cdef}
  2c(q R)^{1/n} \leq  H(q; \bdj) \qquad \textrm{for $\bdj \in \bdJ$ with $R/2 \leq |\bdj| < R$}.
\end{equation}
We can therefore estimate 
\begin{equation*}
    \sum_{\substack{1 \leq R \leq 2M_\gamma q^{\varepsilon}\delta^{-1} \\ R\, \mathrm{dyadic}}}\sum_{\substack{\bdj \in \bdJ: \\ R/2 \leq |\,\bdj\,| < R}}    H(q; \bdj)^{-1}\ll \sum_{\substack{1\leq R\leq 2M_\gamma q^{\varepsilon}\delta^{-1} \\ R\,\mathrm{dyadic}}}  \sum_{\substack{c(qR)^{1/n} \leq \lambda  \\ \lambda\, \mathrm{dyadic}}} \#S^{q,R}( \lambda)\cdot \lambda^{-1}.
\end{equation*}
Here we have used the definition of the sublevel set from \eqref{eq: sublevel} to bound 
\begin{equation*}
  H(q; \bdj)^{-1} \leq 2\lambda^{-1} \qquad \textrm{for all $\bdj \in S^{q,R}( \lambda)$}.
\end{equation*}
Further, we have $$H(q; \bdj)^{-1} \leq (qR)^{-1/2}$$ for $\bdj\in \bdJ$ with $R/2\leq |\bdj|< R$ not contained in the set
$$ \bigcup_{\substack{c(qR)^{1/n} \leq \lambda\leq (qR)^{1/2}  \\ \lambda\, \mathrm{dyadic}}} S^{q,R}(\lambda).$$
Dividing the dyadic summation range in $\lambda$ into two regimes,
\begin{equation}
    \label{eq red 4}
    \Sigma_{\cM}(q, \delta) \ll \Sigma_{\cM}^{\mathrm{near}}(q, \delta) + \Sigma_{\cM}^{\mathrm{far}}(q, \delta) 
\end{equation}
where
\begin{align}
\label{eq def near}
    \Sigma_{\cM}^{\mathrm{near}}(q, \delta) &:= \delta^{n-1} q \sum_{\substack{1\leq R\leq 2M_\gamma q^{\varepsilon}\delta^{-1} \\ R\,\mathrm{dyadic}}}  \sum_{\substack{c(qR)^{1/n} \leq \lambda \leq (qR)^{1/2} \\\lambda\,  \mathrm{dyadic}}} \#S^{q,R}( \lambda)\cdot  \lambda^{-1},\, \textrm{ and } \\
    \Sigma_{\cM}^{\mathrm{far}}(q, \delta) &:= \delta^{n-1}  q \sum_{\substack{1\leq R\leq 2M_\gamma q^{\varepsilon}\delta^{-1} \\ R\,\mathrm{dyadic}}}  R^n (qR)^{-1/2}.\nonumber
\end{align}
 \begin{remark} The following remarks are intended to justify the \textit{near} and \textit{far} labels used in the above notation.
\begin{enumerate}[i)]
    \item The term $\Sigma_{\cM}^{\mathrm{near}}(q, \delta)$ comprises of contributions from $\cI_{\cM}(q; \bdj)$ with relatively weak decay, including those oscillatory integrals with the weakest possible decay $|\cI_{\cM}(q; \bdj)| \ll (qR)^{-1/n}$, as dictated by the van der Corput bound. We shall see below that there are very few lattice points $\bdj$ for which this weakest possible decay rate cannot be improved. Indeed, all such lattice points must lie \textit{near} a $2$-dimensional cone $\Gamma^2$ in $\R^n$. We refer to $\Gamma^2$ as the \textit{slow decay cone}. Moreover, the sets $S^{q,R}( \lambda)$ all lie in some neighbourhood of (that is, \textit{near} to) the slow decay cone, where the size of this neighbourhood is dictated by the $\lambda$ parameter.
    \item By contrast, $\Sigma_{\cM}^{\mathrm{far}}(q, \delta)$ comprises of contributions from $\cI_{\cM}(q; \bdj)$ with the relatively strong decay $|\cI_{\cM}(q; \bdj)| \ll (qR)^{-1/2}$. The corresponding lattice points $\bdj$ lie \textit{far} from the slow decay cone.     
\end{enumerate}
\end{remark}

Clearly the far term satisfies
\begin{equation}
\label{eq far est}
    \Sigma_{\cM}^{\mathrm{far}}(q, \delta) \ll q^{n\varepsilon}\delta^{n-1}q\, (q^{-\frac{1}{2n-1}} \delta^{-1} )^{\frac{2n-1}{2}},
\end{equation}
which is favourable for our purposes.

\begin{remark}
    Observe that we already have the estimate $|\Sigma_{\cM}^{\mathrm{far}}(q, \delta)| \ll \delta^{n-1}q$ for $\delta > q^{-\frac{1}{2n-1}+\nu}$, which is much better than the one guaranteed by Theorem~\ref{thm main}. This is not surprising since, as explained above, the contributions to the term $\Sigma_{\cM}^{\mathrm{far}}(q, \delta)$ come from oscillatory integrals which enjoy strong decay. For $n=3$, this numerology almost matches that of the range $\delta>q^{-\frac{1}{5}}(\log q)^{-\frac{2}{5}}$ from \cite{huangintegral19}.    
\end{remark}





\section{Sublevel sets of the $H$-functional}
\label{sec sublevel set H func}



For each $\lambda > 0$, the sublevel set $S^{q,R}(\lambda)$ as defined in \eqref{eq: sublevel} is contained in the union 
$$\bigcup_{t\in [-2,2]}\big\{ \bdj\in\bdJ : R/2\leq |\bdj|<R \textrm{ and } |\bdj \cdot \gamma^{(r)}(t)| < q^{-1}\lambda^r,\, 1 \leq r \leq n \big\}. $$
Our goal here is to obtain a more detailed description of the geometry of the level sets $S^{q,R}( \lambda)$; for this, we adapt the Frenet frame formalism used in \cite{PS2007, BGHS2021}.



\subsection{The $(n-1)$-normal cone} Let $\bde_1, \dots, \bde_n$ denote the Frenet frame associated to the curve $\cM$; that is, let $$\{\bde_1(t), \dots, \bde_n(t)\}$$ be the orthonormal system obtained from applying the Gram--Schmidt process to the vectors $\gamma^{(1)}(t), \dots, \gamma^{(n)}(t)$. We recall the \textit{Frenet--Serret} equations
\begin{align*}
\bde_1'(t) &= \tilde{\kappa}_1(t) \bde_2(t),\\
\bde_r'(t) &= -\tilde{\kappa}_{r-1}(t) \bde_{r-1}(t) + \tilde{\kappa}_r(t) \bde_{r+1}(t) \qquad \textrm{for $2 \leq r \leq n-1$,} \\ 
\bde_n'(t) &= -\tilde{\kappa}_{n-1}(t) \bde_{n-1}(t)
\end{align*}
where $\tilde{\kappa}_1, \dots, \tilde{\kappa}_{n-1}$ are formed by taking the product of the curvature functions associated to $\gamma$ with the speed factor $|\gamma'(t)|$.\medskip 

For each $t \in U$, there exists some $2 \leq k \leq n$ such that $\bde_{n,k}(t) \neq 0$, where $\bde_{n,k}(t)$ denotes the $k$th component of $\bde_n(t)$. Indeed, otherwise $\bde_n(t)$ corresponds to $e_1$, the first coordinate basis vector, and the graph parametrised form of $\gamma$ implies that $\bde_n(t) \cdot \gamma'(t) = 1$. However, this contradicts the definition of the Frenet frame.\medskip

Using the above observation, together with continuity of the Frenet frame and compactness arguments, we can decompose the interval $[-2,2]$ into a finite union of open intervals $J$ with the following property: for each $J$ there exists $2 \leq k \leq n$ such that $\bde_{n,k}(t) \neq 0$ for all $t \in J$, where $\bde_{n,k}(t)$ denotes the $k$th component of $\bde_n(t)$. By passing to a local portion of the curve, scaling, translating and relabelling the coordinates, we may assume without loss of generality that $\bde_{n,n}(t) \neq 0$ for all $t \in U$. Note that it is possible to preserve the graph parametrised form of our curve under this reduction.\medskip

We define functions $G \colon [-2,2] \to \mathbb{R}^n$ and $g \colon [-2,2] \to \mathbb{R}^{n-1}$ by
\begin{equation}
    \label{eq G def}
    G(t) :=
    \frac{1}{\bde_{n,k}(t)} \bde_n(t), \qquad G(t)=
     \begin{bmatrix} 
     g(t) \\ 1 
     \end{bmatrix}.
\end{equation}
The following observation appeared in \cite{BGHS2021}, but we shall briefly recall the proof here.

\begin{lemma}\label{lem: nondeg} The map $g \colon [-2,2] \to \R^{n-1}$ parametrises a nondegenerate curve, in the sense that
\begin{equation*}
    \det
    \begin{pmatrix}
        g^{(1)}(t) & \cdots & g^{(n-1)}(t)
    \end{pmatrix}
    \neq 0
    \qquad 
    \textrm{for all $t \in [-2,2]$.}
\end{equation*}  
\end{lemma}
\begin{proof}
 Let $k\in \{1, \ldots, n\}$. It follows from the Frenet--Serret equations above that
\begin{equation}
\label{eq: nondeg 1}
    G^{(r)}(t) \in \mathrm{span}\big\{ \bde_{n-r}(t), \dots, \bde_n(t) \big\} \qquad \textrm{for $1 \leq r \leq n-1$.}
\end{equation} 
Indeed, it is not difficult to show that
\begin{equation}\label{eq: nondeg 2}
G^{(r)}(t) \cdot \bde_{n-r}(t) = (-1)^r \prod_{j = n-r}^{n-1} \kappa_{j}(t) \bde_{n,n}(t)^{-1} \neq 0.
\end{equation}
Thus, if we express the vectors $G^{(r)}(t)$ in terms of the orthonormal basis formed by the $\bde_r(t)$, then \eqref{eq: nondeg 1} implies that the $n \times (n-1)$ matrix formed by these vectors is triangular with the nonzero elements \eqref{eq: nondeg 2} along the anti-diagonal. Thus, it has full rank. Consequently, since 
 \begin{equation*}
     G^{(r)}(t) = \begin{bmatrix} 
     g^{(r)}(t) \\ 0 
     \end{bmatrix}
     \qquad \textrm{for $1 \leq r \leq n-1$,}
 \end{equation*}
the desired result follows.
\end{proof}



\subsection{Reinterpreting the sublevel sets} 
\label{ss reinterp sublevel set}
For $t\in [-2,2]$, set
$$\pi_{\mu}^{q, R}(t):=\big\{ \bdj\in\bdJ : R/2\leq |\bdj|<R \textrm{ and } |\bdj\cdot\bde_r(t)| < C_\gamma q^{-1}\lambda^r,\, 1 \leq r \leq n \big\}. $$
Provided $C_\gamma \geq 1$ is chosen sufficiently large, 
\begin{equation*}
    S^{q,R}( \lambda) \subseteq \bigcup_{t\in [-2,2]} \pi^{q, R}(t);
\end{equation*}
indeed, this follows by noting $\bde_r(t)$ lies in the linear span of $\gamma^{(1)}(t), \ldots, \gamma^{(r)}(t)$.

Now, for any $\xi = (\xi_1, \dots, \xi_n) \in \widehat{\R}^n$ and $t\in [-2,2]$, we know that $\xi - \xi_n G(t) \in \widehat{\R}^{n-1} \times \{0\}$. By Lemma~\ref{lem: nondeg}, there exist coefficients $\rho_1, \dots, \rho_{n-1} \in \R$ such that
\begin{equation*}
    \xi - \xi_n G(t) = \sum_{r = 1}^{n-1} \rho_r G^{(r)}(t).
\end{equation*}
For each $1 \leq r \leq n-1$, we form the inner product of both sides of the above identity with the Frenet vector $\bde_r(t)$. Combining the resulting expressions with the linear independence relations inherent in \eqref{eq: nondeg 1}, the coefficients $\rho_r$ can be related to the numbers $\xi \cdot \bde_r(t)$ via a lower anti-triangular transformation, viz. 
\begin{equation*}
    \begin{bmatrix}
    \xi \cdot \bde_1(t) \\
    \vdots \\
    \xi \cdot \bde_{n-1}(t)
    \end{bmatrix}
    =
    \begin{bmatrix}
    0 & \cdots & G^{(n-1)}(t) \cdot \bde_1(t) \\
    \vdots & \ddots & \vdots \\
    G^{(1)}(t) \cdot \bde_{n-1}(t) & \cdots & G^{(n-1)}(t) \cdot \bde_{n-1}(t)
    \end{bmatrix}
    \begin{bmatrix}
    \rho_1 \\
    \vdots \\
     \rho_{n-1}
    \end{bmatrix}.
\end{equation*}
Now for $\bdj\in \pi^{q, R}(t)$ with $t\in [-2,2]$,
using the fact that the inverse of a nonsingular lower anti-triangular matrix is an upper anti-triangular matrix, it follows that
$$|\rho_r| \ll |\bdj\cdot\bde_r(t)| \ll q^{-1}\lambda^{n-r}$$ for $1 \leq r \leq n-1$.

It thus follows that for $t\in [-2,2]$, each set $\pi^{q, R}(t)$ is contained in a set $\alpha^{q, R}(t)$ consisting of all $ \bdj \in \bdJ$ for which 
\begin{equation}
\label{eq sub sub 2}
\bdj - j_n G(t) = \sum_{r = 1}^{n-1} \rho_r G^{(r)}(t) \quad \textrm{and} \quad |j_n| \leq R 
\end{equation}
hold for some $\rho_r \in \R$ satisfying $|\rho_r| \ll \min\left\{R, \frac{\lambda^{n-r}}{q}\right\}$ for $1\leq r\leq n-1$. In particular,
\begin{equation}\label{eq: alpha union}
    S^{q,R}( \lambda) \subseteq \bigcup_{t\in [-2,2]} \alpha^{q, R}(t).
\end{equation}
Thus, our task boils down to estimating the cardinality of the union appearing on the right-hand side of \eqref{eq: alpha union}.




\subsection{Volume bounds}

For $G$ as in \eqref{eq G def} and $T := (T_0, T_1, \dots, T_{n-2}) \in [1, \infty)^{n-1}$, we define 
\begin{equation*}
    \cF(T) := \left\{\sum_{r=0}^{n-2} s_rG^{(r)}(t): (t, s_0, \dots, s_{n-2}) \in \Omega(T) \right\},
\end{equation*}
where
\begin{equation*}
    \Omega(T) := [-2,2]\times \omega(T) \quad \textrm{for} \quad \omega(T) := [-T_0, T_0] \times \cdots \times [-T_{n-2}, T_{n-2}].
\end{equation*}

The frequencies contributing to the sum in \eqref{eq def near} can be decomposed as sublevel sets of the $H$-functional, based on its size $\lambda$. From the discussion in \S\ref{ss reinterp sublevel set}, and in particular \eqref{eq: alpha union}, it follows that the sublevel sets of the $H$-functional can be reinterpreted in terms of the set $\cF(T)$ with
$$T_0 \sim R,\, \qquad T_r \sim \max\left\{1, \min\left\{R, \frac{\lambda^{n-r}}{q}\right\}\right\}\qquad \textrm{for } 1\leq r\leq n-2.$$
To obtain an upper bound on the cardinality of these sublevel sets, we are interested in the number of lattice points contained within an $O(1)$-neighbourhood of $\cF(T)$. In view of this, we let $\cF_{\sigma}(T)$ denote the $\sigma$-neighbourhood of the set $\cF(T)$ for $\sigma > 0$.  Our lattice point count is a consequence of the following simple volume estimate.

\begin{lemma}
\label{lem triv count}
For $R \geq T_0, \ldots, T_{n-2}\geq \sigma \geq 1$ and $T=(\tvec)$, we have
\begin{equation*}
  \mathrm{Vol}_n(\cF_{\sigma}(T)) \ll_{\sigma} R\prod_{r=0}^{n-2} T_r,  
\end{equation*}
where $\mathrm{Vol}_n$ denotes the $n$-dimensional Lebesgue measure.
\end{lemma}

\begin{proof} The set $\cF(T)$ is parametrised by the mapping 
\begin{equation*}
 \phi \colon \Omega(T) \to \R^n, \qquad \phi(t, \svec):= \sum_{r=0}^{n-2} s_rG^{(r)}(t).
\end{equation*}
Let $\mu$ denote the push-forward of the Lebesgue measure on $\Omega(T)$ under $\phi$, so that
\begin{equation*}
    \int_{\R^n} F \,\ud \mu := \int_{\Omega(T)} F\circ \phi(u)\,\ud u \qquad \textrm{for all $F \in C_c(\R^n)$}
\end{equation*}
and
\begin{equation}\label{eq: triv count 1}
    \mu(\R^n) = \textrm{Vol}_{n}\left(\Omega(T)\right).
\end{equation}
We note that $\mu$ is \textit{not} comparable with the $n$ dimensional Lebesgue measure on $\cF(T)$, due to the singular behaviour of the map $\phi$.\medskip

For $t \in [-2,2]$, consider the map $\phi_t \colon \omega(T) \to \R$ given by $\phi_t(s) := \phi(t, s)$. Then $\phi_t$ is linear with associated matrix 
\begin{equation*}
    D \phi_t := 
    \begin{bmatrix}
        G^{(0)}(t) & \cdots & G^{(n-2)}(t)
    \end{bmatrix}.
\end{equation*}
Furthermore, we can choose some $C_{\gamma} \geq 1$, depending only on $\gamma$, so that 
\begin{equation*}
   \sum_{r=0}^{n-1}|G^{(r)}(t)| \leq C_{\gamma} \qquad \textrm{for all $t \in [-2,2]$.}
\end{equation*}
 This implies that the $\ell^2$ matrix norm satisfies $\|D \phi_t\| \leq C_{\gamma}$ for all $t \in [-2,2]$.\medskip

Let $\chi_{B(0, 2\sigma)}$ denote the characteristic function of the ball $B(0, 2\sigma)$ in $\R^n$ centred at the origin with radius $2\sigma$.  Consider the convolution 
\begin{equation*}
    \chi_{B(0, 2\sigma)} \ast \mu(x) = \int_{\Omega(T)} \chi_{B(0, 2\sigma)}(x - \phi(u))\,\ud u.
\end{equation*}
If $x \in \cF_{\sigma}(T)$, then there exists some $(t_0, s_0) \in \Omega(T)$ such that $|x - \phi(t_0, s_0)| < \sigma$. By our choice of $C_{\gamma}$, we know that
\begin{equation*}
    |\phi(t_0, s) - \phi(t_0, s_0)| \leq C_{\gamma} |s - s_0| \qquad \textrm{for all $s \in \omega(T)$}
\end{equation*}
and 
\begin{equation*}
    |\phi(t_0, s) - \phi(t, s)| \leq C_{\gamma} R |t - t_0| \qquad \textrm{for all $(t,s) \in \Omega(T)$.}
\end{equation*}
In particular, if $|s - s_0| \leq \sigma/(4C_{\gamma})$ and $|t - t_0| < \sigma/(4C_{\gamma} R)$, then $|x - \phi(t, s)| < 2\sigma$. From this observation, we see that
\begin{equation*}
    \chi_{B(0, 2\sigma)} \ast \mu(x) \gg_{\sigma} R^{-1} \chi_{\cF_{\sigma}(T)}(x).
\end{equation*}
By the Fubini--Tonelli theorem,
\begin{equation*}
\mathrm{Vol}_n(\cF_{\sigma}(T)) \ll_{\sigma} R \int_{\R^n} \chi_{B(0, 2\sigma)} \ast \mu(x)\,\ud x = R \mu(\R^n) \|\chi_{B(0, 2\sigma)}\|_{L^1(\R^n)}.
\end{equation*}
Using \eqref{eq: triv count 1}, we conclude that
\begin{equation*}
  \mathrm{Vol}_n(\cF_{\sigma}(T))   \ll_{\sigma} R\, \textrm{Vol}_{n}\left(\Omega(T)\right) \ll R \prod_{r = 0}^{n-2} T_r,
\end{equation*}
as required. 
\end{proof}

For $T \in [1, \infty)^{n-1}$ and $\sigma > 0$, we are interested in the associated set of lattice points
\begin{equation*}
    \mathcal{L}_{\sigma}(T):= \cF_{\sigma}(T)\cap \mathbb{Z}^n.
\end{equation*}
The cardinality of this set can be crudely estimated using Lemma~\ref{lem triv count}.

\begin{corollary}
\label{cor triv count}
For $R \geq T_0, \ldots, T_{n-2}\geq \sigma \geq 1$ and $T=(\tvec)$, let $\mathcal{L}_{\sigma}(T)$ be as defined above. Then
    $$\# \mathcal{L}_{\sigma}(T)\ll_{\sigma} R \prod_{r=0}^{n-2} T_r.$$
\end{corollary}

\begin{proof} For every $a \in \cL_{\sigma}(T)$, the euclidean ball $B(a, 1/4)$ is contained in $\cF_{2 \sigma}(T)$. By disjointedness and the volume bound from Lemma~\ref{lem triv count} we therefore have
\begin{equation*}
    \#\cL_{\sigma}(T) \sim \sum_{a \in \cL_{\sigma}(T)} \mathrm{Vol}_n(B(a, 1/4)) \leq  \mathrm{Vol}_n(\cF_{2\sigma}(T)) \ll_{\sigma}  R \prod_{r=0}^{n-2} T_r,
\end{equation*}
as required. 
\end{proof}




\section{Proof of theorem~\ref{thm main}}
\label{sec proof main}

We now have all ingredients in place to prove our main theorem.

\begin{proof}[Proof (of Theorem~\ref{thm main})]
For any $\varepsilon>0$, by combining \eqref{eq red 2}, \eqref{eq def sigma}, \eqref{eq red 4} and \eqref{eq far est}, we obtain
\begin{align}
\nonumber
    \left|\fkN_{\cM}(q, \delta)-c_0\delta^{n-1} q\right| &\ll_{\varepsilon} \Sigma_{\cM}^{\textrm{near}}(q, \delta)+\delta^{n-1}q^{1+n\varepsilon}\, (q^{-\frac{1}{2n-1}} \delta^{-1} )^{\frac{2n - 1}{2}} +O_{\varepsilon}(1) \\
  \nonumber
    &= \Sigma_{\cM}^{\textrm{near}}(q, \delta)+\delta^{n-1}q^{1+n\varepsilon}\, (q^{-\kappa(n-1)} \delta^{-1} )^{n-1 +1 - \frac{1}{n-(n-1)+1}} \\
      \label{eq start est}
    & \qquad\qquad +O_{\varepsilon}(1),
\end{align}
where $\kappa(n-1) := \frac{1}{2n-1}$ and $c_0$ is positive with $c_0\sim 1$. Here $\Sigma_{\cM}^{\textrm{near}}(q, \delta)$ is as defined in \eqref{eq def near}. We express the right-hand exponents in the above form for consistency with later estimates. 

Let $\Lambda_n^{q, R}$ denote the set of all dyadic $\lambda$ satisfying $c(qR)^{1/n} \leq \lambda \leq (qR)^{1/2}$, where $c$ is as in \eqref{eq cdef}. 
Recall the definition of the level set $S^{q, R}(\lambda)$ from \eqref{eq: sublevel}. If $\bdj \in S^{q, R}(\lambda)$, then by \eqref{eq: alpha union}, there exists some $t\in [-2,2]$ such that $\bdj \in \alpha^{q, R}(t)$, as defined in \eqref{eq sub sub 2}. Consequently, $|j_n|\leq R$, and there exist $\rho_r \in \R$, with $|\rho_r|\ll R\min\left\{1, \frac{\lambda^{n-r}}{qR}\right\}$ for $1 \leq r \leq n-1$, such that 
\begin{equation*}
    \bdj = j_nG(t) + \sum_{r = 1}^{n-1} \rho_r G^{(r)}(t).
\end{equation*}
For a suitably large constant $C>1$ (depending on $\gamma$ and $n$), set
\begin{equation*}
    \rho_n^{q, R}(r; \lambda) := 
    \begin{cases}
    C R\min\Big\{1, \frac{\lambda^{n-r}}{q R}\Big\} & \textrm{ for } 1\leq r\leq n-1,\\
    1 & \textrm{ for } r=n.
    \end{cases}
\end{equation*}

For $\lambda\in \Lambda_n^{q, R}$, let $1 \leq D(\lambda) \leq n$ be the smallest value of $r$ such that $\rho_n^{q, R}(r; \lambda) \leq 1$. For $1\leq d\leq n$, define
\begin{equation*}
   \Lambda_n^{q, R}(d) := \{ \lambda \in \Lambda_n^{q, R} : D(\lambda) = d \}.  
\end{equation*}
For $\lambda\in \Lambda_n^{q, R}(n)$, we trivially estimate
\begin{equation}
    \label{eq sqrn est}
    \#S^{q, R}(\lambda)\ll R^n.
\end{equation}
In this case, we also have $\rho_n^{q, R}(n-1; \lambda)=C R\min\Big\{1, \frac{\lambda}{q R}\Big\}\geq 1$, which gives the lower bound
\begin{equation}
    \label{eq lbd lb n}
    \lambda\gg q.
\end{equation}
Otherwise, let $\lambda\in \Lambda_n^{q, R}(d)$ with $1\leq d\leq n-1$. Then we can write 
\begin{equation*}
    \bdj = j_nG(t) + \sum_{r = 1}^{n-2} \rho_r G^{(r)}(t)+O(1),
\end{equation*}
with
\begin{equation*}
|\rho_r|\leq \begin{cases}
    R &\textrm{ for } 1\leq r\leq d-2,\\
    CR\min\left\{1, \frac{\lambda^{n-d+1}}{qR}\right\} &\textrm{ for } r=d-1,\\
    1 &\textrm{ for } d\leq r\leq n-2.
\end{cases}    
\end{equation*}
Since $|j_n|\leq R$, for $\lambda\in \Lambda_n^{q, R}(d)$ with $1\leq d\leq n-1$, the set
$S^{q, R}(\lambda)$ is contained in $\cF_{\sigma}(T)$, where $\sigma>0$ is a constant depending only on $\gamma$ and $n$, and
$$T=\left(R, \underbrace{R,\ldots, R, }_{d-2 \textrm{ times}} C R\min\left\{1, \frac{\lambda^{n-d+1}}{qR}\right\}, 1, \ldots, 1\right).$$
An application of Corollary~\ref{cor triv count} therefore yields
\begin{align*}
    \#S^{q, R}(\lambda) &\ll \#\cL_{\sigma}\left(R, \underbrace{R, \ldots, R, }_{d-2 \textrm{ times}} C R\min\left\{1, \frac{\lambda^{n-d+1}}{qR}\right\}, 1, \ldots, 1\right)\\
    &\ll R^{d+1}\min\left\{1, \frac{\lambda^{n-d+1}}{qR}\right\}.
\end{align*}
 Thus,
\begin{equation*}
   \sum_{\lambda \in \Lambda_n^{q, R}(d)} \# S^{q, R}(\lambda)\cdot\lambda^{-1} \ll R^{d+1} \sum_{\lambda \in \Lambda_n^{q, R}(d)} \min\Big\{1, \frac{\lambda^{n-d+1}}{q R}\Big\} \lambda^{-1}.
\end{equation*}
Splitting the sum in $\lambda$, we have
\begin{equation*}
\sum_{\substack{\lambda \in \Lambda_n^{q, R}(d): \\ \lambda \geq (qR)^{\frac{1}{n-d+1}}}} \lambda^{-1}  + \sum_{\substack{\lambda \in \Lambda_n^{q, R} (d): \\ \lambda \leq (qR)^{\frac{1}{n-d+1}}}} \Big(\frac{\lambda^{n-d+1}}{q R}\Big) \lambda^{-1} \ll (qR)^{-\frac{1}{n-d+1}}.
\end{equation*}
For $1\leq d\leq n-1$, using the above, we can estimate, 
\begin{equation*}
 \delta^{n-1} q  \sum_{\substack{1\leq R\leq 2M_\gamma q^{\varepsilon}\delta^{-1} \\ R\,\mathrm{dyadic}}}  \sum_{\lambda \in \Lambda_n^{q, R}(d)} \# S^{q, R}(\lambda)\cdot\lambda^{-1} \ll \delta^{n-1}q^{1+n\varepsilon} \cdot\big( q^{-\frac{1}{n-d+1}} \delta^{-(d+1 - \frac{1}{n-d+1})}\Big).
\end{equation*}
On the other hand, for $\lambda\in \Lambda_n^{q, R}(n)$, using \eqref{eq sqrn est} and \eqref{eq lbd lb n}, we get
\begin{align}
    \nonumber
   \delta^{n-1} q  \sum_{\substack{1\leq R\leq 2M_\gamma q^{\varepsilon}\delta^{-1} \\ R\,\mathrm{dyadic}}}\sum_{\lambda \in \Lambda_n^{q, R}(n)} \# S^{q, R}(\lambda)\cdot\lambda^{-1} &\ll \delta^{n-1} q^{1+n\varepsilon}\big(q^{-1/n}\delta^{-1} \big)^n \\
   \label{eq n fin est}
   &= \delta^{n-1} q^{1+n\varepsilon}\big(q^{-\kappa(n)}\delta^{-1} \big)^{n +1 - \frac{1}{n-n + 1}} 
\end{align}
for $\kappa(n) := 1/n$. We express the right-hand exponents in the above form for consistency with the estimates below. 

Summing over $1\leq d\leq n$, we can estimate
\begin{align}
    \Sigma_{\cM}^{\mathrm{near}}(q, \delta) &= \delta^{n-1} q^{1+n\varepsilon} \sum_{\substack{1 \leq R \leq 2\delta^{-1} \\ R\,\mathrm{dyadic}}}  \sum_{\substack{c(qR)^{1/n} \leq \lambda \leq (qR)^{1/2} \\\lambda\, \mathrm{dyadic}}} \#S^{q,R}( \lambda)\cdot  \lambda^{-1}\nonumber\\
    & \ll \delta^{n-1} q^{1+n\varepsilon} \max_{1 \leq d \leq n} \big(q^{-\kappa(d)} \delta^{-1}\big)^{d+1 - \frac{1}{n-d+1}},
    \label{eq near est}
\end{align}
where 
\begin{equation*}
    \kappa(d):=\frac{1}{(n-d+1)(d+1)-1} \qquad \textrm{for $1 \leq d \leq n$.}
\end{equation*}
Note this definition is consistent with \eqref{eq start est} and \eqref{eq n fin est}. Thus, from \eqref{eq start est} and \eqref{eq near est}, we get
\begin{equation}
    \label{eq fin est 1}
     \left|\fkN_{\cM}(q, \delta)-c_0\delta^{n-1} q\right|\ll_{\varepsilon} \delta^{n-1} q^{1+n\varepsilon}\, \cE(q, \delta) + O_{\varepsilon}(1),
\end{equation}
where
\begin{equation}
    \label{eq error tm}
    \cE(q, \delta):= \max_{1 \leq d \leq n} \big(q^{-\kappa(d)} \delta^{-1}\big)^{d+1 - \frac{1}{n-d+1}}.
\end{equation}

The polynomial $-X^2 + (n+1)X +n$ has a unique critical point at $X = \frac{n+1}{2}$, which is a global maximum. Hence, if $n$ is odd, then 
\begin{equation*}
    \kappa(d) \leq \kappa\Big(\frac{n+1}{2}\Big)= \frac{4}{n(n+4) - 1} =: \beta(n) \qquad \textrm{for $1 \leq d \leq n$}\,,
\end{equation*}
while if $n$ is even, we have 
\begin{equation*}
    \kappa(d) \leq \kappa\Big(\frac{n}{2}\Big) = \kappa\Big(\frac{n+2}{2}\Big)=\frac{4}{n(n+4)} =: \beta(n) \qquad \textrm{for $1 \leq d \leq n$.}
\end{equation*}
Thus, recalling \eqref{eq error tm}, we have
\begin{equation}\label{eq E bound}
 q^{n\varepsilon}\cE(q, \delta)=O(q^{-\varepsilon}) \qquad \textrm{for all} \qquad \delta\geq q^{-\beta(n)+(n+1)\varepsilon}.
\end{equation}

To prove the lower bound in Theorem~\ref{thm main}, let $\nu>0$ be given. By choosing $\varepsilon := \nu/(n+1)$ in \eqref{eq fin est 1} and applying \eqref{eq E bound}, we obtain
\begin{equation*}
    \fkN_{\cM}(q, \delta)\gg_{\nu} \delta^{n-1} q  \qquad \textrm{for $\delta\geq q^{-\beta(n)+\nu}$,}
\end{equation*}
provided $q \gg_{\nu} 1$ is sufficiently large. We now take $w = w^-$ and $\eta=\eta^{-}$ (see \S \ref{ss smooth} for their definitions). For this choice of $w$ and $\eta$, and for $\delta\geq q^{-\beta(n)+\nu}$, we have
$$N_{\cM}(q, \delta)\geq \fkN_{\cM}(q, \delta)\gg_{\nu} \delta^{n-1} q\,,$$
which establishes \eqref{eq main lb} for $q \gg_{\nu} 1$.\medskip

To prove the upper bound in Theorem~\ref{thm main}, again let $\nu > 0$ be given. This time, we take $w = w^+$ and $\eta=\eta^+$. Then for any $\delta\in (0, 1/4)$ and $q\geq 1$, using \eqref{eq fin est 1},
we obtain
\begin{equation*}
    N_{\cM}(q, \delta)\leq \fkN_{\cM}(q, \delta)\ll_{\varepsilon} \delta^{n-1} q\left(1+q^{n\varepsilon} \cE(q, \delta)\right)+O_{\varepsilon}(1)
\end{equation*}
for all $\varepsilon > 0$. Observe that for a fixed $q$, the term $\cE(q, \delta)$ is decreasing as a function of $\delta$. We now distinguish two cases. 
When $\delta \leq  q^{-\beta(n)+(n+1)\varepsilon}$, since $N_{\cM}(q, \delta)$ is a non-increasing function of $\delta$, using \eqref{eq E bound} we can estimate
$$N_{\cM}(q, \delta)\ll N_{\cM}(q, q^{-\beta(n)+(n+1)\varepsilon})\ll_{\varepsilon} \left(q^{-\beta(n)+(n+1)\varepsilon}\right)^{n-1}q\leq q^{\Theta(n)+n^2\varepsilon},$$
where $\Theta(n)$ is as defined in \eqref{def theta}. Here we have used the identity $1 - (n-1)\beta(n) = \Theta(n)$.  On the other hand,  if $\delta\geq q^{-\beta(n)+(n+1)\varepsilon}$, then we can again use \eqref{eq fin est 1} and \eqref{eq E bound} to estimate
\begin{equation*}
    N_{\cM}(q, \delta) \ll_{\varepsilon} \delta^{n-1} q.
\end{equation*}
Thus, choosing $\varepsilon := \nu / n^2$, we can conclude in both cases that
$$N_{\cM}(q, \delta)\ll \delta^{n-1} q+ q^{\Theta(n)+\nu}.$$
This establishes \eqref{eq main est}, and finishes the proof.
\end{proof}




\bibliography{Reference}
\bibliographystyle{amsplain}

\end{document}